\title{Consequences of Matrices Sharing Eigenvalues and Eigenvectors}
\author{Kayla Miller\thanks{Mount Greylock and Williams College}
        \and
        Steven J. Miller\thanks{Williams College}
        \and
        Fuzhen Zhang\thanks{Nova Southeastern University}               
        }
\documentclass{article}
\usepackage{graphicx}
\usepackage{amsbsy}

\usepackage{thmtools}
\usepackage{times}
\usepackage[T1]{fontenc}
\usepackage{mathrsfs}
\usepackage{latexsym}
\usepackage{flafter,epsf}
\usepackage{amsmath,amsfonts,amsthm,amssymb,amscd}
\usepackage{color}
\usepackage{fix-cm}
\usepackage{hyperref}
\hypersetup{colorlinks=true,linkcolor=blue}
\usepackage{color}
\usepackage{url}

\newtheorem{theorem}{Theorem}

\newtheorem{rek}{Remark}
\newtheorem{exa}{Example}
\newtheorem{corollary}{Corollary}

\newcommand{\vv}{v}

\newcommand{\twocase}[5]{#1 \begin{cases} #2 & \text{#3}\\ #4
&\text{#5} \end{cases}   }

\newcommand{\mattwo}[4]
{\left[\begin{array}{cc}
                        #1  & #2   \\
                        #3 &  #4
                          \end{array}\right] }
\newcommand{\vect}[1]{#1}

\newcommand{\vectwo}[2]
{\left[\begin{array}{c}
                        #1   \\
                        #2
\end{array}\right] }

\newcommand{\matthree}[9]
{\left[\begin{array}{ccc}
                        #1  & #2 & #3  \\
                        #4 &  #5 & #6 \\
                        #7 &  #8 & #9
                          \end{array}\right] }

\newcommand{\vecfour}[4]
{\left[\begin{array}{c}
                        #1   \\
                        #2   \\
                        #3   \\
                        #4
\end{array}\right] }

\begin{document}
\newpage
\maketitle
\begin{abstract}
While studying for a linear algebra final, the first named author prepared some test questions for herself to see how well she understood the material, and asked the second named author: \emph{If $A$ and $A^T$ have the same eigenvalues and eigenvectors, is $A$ a symmetric matrix?} We show how this excellent question is a great springboard to related questions, in particular when do equal eigenvalues and eigenvectors imply the matrices are, if not equal, at least closely related (such as similar or the transpose/complex conjugate transpose of each other)? The answer depends on how we interpret the question, and provides a great opportunity to talk about creating good questions. In particular, we characterize matrices $A$ for which the transpose or conjugate transpose shares the same eigenvectors (regardless of eigenvalues) and, for each eigenvalue, the same eigenpair (equivalently, the same eigenspace). Thus,  a square matrix $A$ is Hermitian if and only if $A^*$ has the same eigenpairs as $A$; moreover, if $A$ is a real matrix with real eigenvalues and $A^T$ has the same eigenvectors as $A$, then $A$ is symmetric.
\end{abstract}




\section{Introduction}

\begin{quote} \emph{Our similarities bring us to a common ground; our differences allow us to be fascinated by each other.} --Tom Robbins
\end{quote}

\medskip

While the first named author was reviewing linear algebra concepts for her final exam,
she thought of questions that could extend her  understanding. One of the most important topics in a first-year course is eigenvalues and eigenvectors. Recall that a non-zero vector $v$ is an eigenvector of an $n \times n$ matrix $A$ if $A\vv = \lambda \vv$ for some $\lambda \in \mathbb{C}$. The standard way to compute eigenvalues  is through the determinant equation\footnote{For large matrices this is a terrible approach, as it leads to solving for the roots of a polynomial of degree $n$.}: if $A\vv = \lambda \vv$ then $(A - \lambda I)\vv = \vect{0}$, which means that $A - \lambda I$ is a singular matrix and hence has determinant zero. Thus while $A$, its transpose $A^T$, and its complex-conjugate transpose $A^\ast$ (in case of real eigenvalues)    all have the same eigenvalues, they need not have the same eigenvectors.
\medskip

She asked the second author the following question: \emph{If $A$ and $A^T$ have the same eigenvectors must they be the same? In other words, does sharing the same eigenvalues and eigenvectors force $A$ to be symmetric?}

\medskip
As with many questions, it's possible to implicitly assume additional properties that are not  explicitly stated. When told $A$ and $A^T$ have the same eigenvalues and eigenvectors, does this mean that there is a full set of $n$ linearly independent eigenvectors? Does it mean that a shared eigenvector of $A$ and $A^T$ corresponds to the same eigenvalue, or perhaps while the two matrices have the same eigenvalues and eigenvectors, they may not have the same \emph{pairs} of eigenvalues and eigenvectors.
\medskip

While in conversations this was clarified to having the eigenvectors associated with the same eigenvalues (though not necessarily having $n$ linearly independent eigenvectors), the fact that the original question could be interpreted in multiple different ways, depending on what unspoken assumptions the listener made, led to the interesting questions which we pursue below (the second named author heard a talk by the third named author at the 27th Conference of the International Linear Algebra Society at Virginia Tech, and asked him these questions after his talk).
\medskip

We characterize matrices $A$ for which the transpose $A^T$ (or conjugate transpose $A^*$) shares the same eigenvectors as $A$ (regardless of eigenvalues) and, for each eigenvalue, the same eigenpair (equivalently, the same eigenspace). As a result, a square matrix $A$ is Hermitian if and only if $A^*$ has the same eigenpairs as $A$; moreover, if $A$ is a real matrix with real eigenvalues such that $A^T$ has the same eigenvectors as $A$, then $A^T = A$.

\newpage

After giving some examples to build intuition, we prove the following.

\begin{theorem}\label{Thm:Normality}
Every eigenvector  of $A$ is an eigenvector (regardless of the eigenvalues) of $A^*$ if and only if $A$ is \emph{normal},
that is,  $A^*A=AA^*$.
\end{theorem}

\begin{theorem}\label{Thm:Normality2}
If every eigenpair of $A$ is an eigenpair of $A^*$; to be exact,
\begin{equation*}
Av = \lambda v
 \quad \Rightarrow \quad
A^*v = \lambda v,
\end{equation*}
then $A$ is a Hermitian matrix, that is, $A^*=A$ (thus all eigenvalues are real).
\end{theorem}

\section{Same Eigenvalues and Eigenvectors but Not Similar}

We begin with some elementary arguments and examples. Is it possible for two $n \times n$ matrices to have the same eigenvalues \emph{and} the same eigenvectors without being similar\footnote{Matrices $A$ and $B$ are similar if there exists an invertible matrix $S$ such that $S^{-1} A S = B$; a simple computation shows that if $A$ is similar to $B$ then $B$ is similar to $A$.}? At first glance the problem seems trivial. Let the $n$ eigenvalues be $\lambda_1, \dots, \lambda_n$ with corresponding $n$ eigenvectors $\vect{v_1}, \dots, \vect{v_n}$ (column vectors in $\Bbb C^n$), consider the matrix $S=[v_1, \dots, v_n]$,
  and  note that
$$AS=[Av_1, \dots, Av_n]= S\,{\rm diag} (\lambda_1, \dots, \lambda_n].$$
We see that
$$S^{-1} A S \ = \ \matthree{\lambda_1}{}{}{}{\ddots}{}{}{}{\lambda_n} \ = \ S^{-1} B S.$$  So,  $S$ diagonalizes both $A$ and $B$. Thus,  not only are $A$ and $B$ similar but they are equal!
\medskip

Unfortunately, we made a common mistake in this argument:  we assumed that there were $n$ linear independent eigenvectors. This is not always true. If an eigenvalue $\lambda$ has algebraic multiplicity $k$ (which means $x = \lambda$ is a root of multiplicity exactly $k$ of $\det(A - x I)$), it need not have geometric multiplicity $k$; in other words, the subspace of eigenvectors with eigenvalue $\lambda$ need not have dimension $k$.
The simplest example of this is the matrix $$\mattwo{\lambda}{1}{0}{\lambda};$$ $\lambda$ is an eigenvalue with algebraic multiplicity 2 but the eigenspace has dimension 1, spanned by $\vectwo{1}{0}$.
\medskip

What if we assume that $A$ and $B$ have the same eigenvalues with the same multiplicities and the same eigenvectors, which are linear independent -- does our argument hold then, and must the matrices be equal? Sadly no, as in our discussion we assumed something that we were not given, namely that they have the same \emph{eigenpairs}: $$A \vect{v_i} \ = \ \lambda_i \vect{v_i} \ = \ B \vect{v_i}.$$ It's possible for them to have the same eigenvalues and the same eigenvectors, but different eigenpairs.
\medskip

 Consider $$A \ = \ \mattwo{1}{0}{0}{2} \ \ \ {\rm  and} \ \ \ B \ = \ \mattwo{2}{0}{0}{1}.$$ Both have eigenvalues 1 and 2, both have eigenvectors $\vect{v_1} = \vectwo{1}{0}$ and $\vect{v_2} = \vectwo{0}{1}$, but $$A \vect{v_1} \ = \ 1 \vect{v_1} \ \ \ {\rm and} \ \ \ B \vect{v_1} \ = \ 2 \vect{v_1}$$ while $$A \vect{v_2} \ = \ 2 \vect{v_2} \ \ \ {\rm and} \ \ \ B \vect{v_2} \ = \ 1 \vect{v_2}.$$ In this case $A$ and $B$ are clearly not equal, though they are similar\footnote{We leave it to the reader to figure out how to find $S$.}: $$S^{-1} A S \ = \ B \ \ \ {\rm  where} \ \ \ S \ = \ \mattwo{\ \ 0}{1}{1}{0}.$$

We have gone from ``the matrices must be equal" to ``they can be similar but not equal". Can we find an example where the matrices have the same eigenvalues and eigenvectors, but they are not similar? Yes! Consider
\begin{equation}\label{eq:ABsameevvnotsamejcc} A \ = \
\left[\begin{array}{cccc}
0 & 0 & 0 & 0 \\
0 & 0 & 0 & 0 \\
0 & 0 & 1 & 1 \\
0 & 0 & 0 & 1
\end{array}\right] \ \ \ {\rm and}\ \ \ B \ = \ \left[\begin{array}{cccc}
1 & 0 & 0 & 0 \\
0 & 1 & 0 & 0 \\
0 & 0 & 0 & 1 \\
0 & 0 & 0 & 0 \\
\end{array}\right].
\end{equation}

Both matrices have the same set of eigenvalues with the same multiplicities: $\{0, 0, 1, 1\}$. Each of $A$ and $B$ has three eigenvectors:
\begin{equation}
    \vect{v_1} \ = \ \vecfour{1}{0}{0}{0}, \ \ \
    \vect{v_2} \ = \ \vecfour{0}{1}{0}{0}, \ \ \
    \vect{v_3} \ = \ \vecfour{0}{0}{1}{0},
\end{equation}
but they are associated to different eigenvalues; for $A$,  $\vect{v_1}$ and $\vect{v_2}$ have eigenvalue 0 while $\vect{v_3}$ has eigenvalue 1, while for $B$ it's the opposite and $\vect{v_1}$ and $\vect{v_2}$ have eigenvalue 1 while $\vect{v_3}$ has eigenvalue 0. The easiest way to see these two matrices cannot be similar is to note that $A$ and $B$ have different   Jordan Canonical Forms, and similar matrices have the same Jordan Canonical Form.

\begin{rek}[Jordan Canonical Form] Sadly Jordan Canonical Form is typically not covered in most first courses on the subject; while not every matrix can be diagonalized, every matrix can be brought into block diagonal form where each $k \times k$ block is all zeros except  for an eigenvalue with multiplicity $k$ along the main diagonal and 1's on the diagonal immediately above it (see, e.g., \cite[Section\,3.4]{Zh}).
 Thus the first few Jordan blocks by size are \begin{equation} (\lambda), \ \ \ \mattwo{\lambda}{1}{0}{\lambda}, \ \ \  \matthree{\lambda}{1}{0}{0}{\lambda}{1}{0}{0}{\lambda}, \ \ \  \dots. \end{equation} Returning to  \eqref{eq:ABsameevvnotsamejcc}, we have \begin{equation}
    A \ = \ \matthree{[0]}{}{}{}{[0]}{}{}{}{\mattwo{1}{1}{0}{1}}, \ \ \ B \ =\  \matthree{[1]}{}{}{}{[1]}{}{}{}{\mattwo{0}{1}{0}{0}};
\end{equation} both have three Jordan blocks but $A$ has two $1\times 1$ blocks with eigenvalue 0 and one $2\times 2$ block with eigenvalue 1, while   $B$ has
two $1\times 1$ blocks with eigenvalue 1 and one $2\times 2$ block with eigenvalue 0.
\end{rek}

Thus it's  possible for two matrices to have the same eigenvalues and eigenvectors, but if the eigenvectors are not paired with the same eigenvalues the matrices need not be similar. Armed with the experience from these examples, we now turn to proving our theorems on the characterization of paired matrices based on certain properties of their eigenvalues and eigenvectors.



\section{Proofs of Main Results}

\subsection{Preliminaries}

We first recall some notation. As always, $A$ is an $n\times n$ real or complex matrix. If $Av=\lambda v$ for some scalar $\lambda \in \mathbb{C}$ (the field of complex numbers) and some nonzero vector
$v\in \mathbb{C}^n$ (the complex $n$-tuples written as  column vectors), then $\lambda$ is called an eigenvalue of $A$, $v$ is the eigenvector of $A$ associated with the eigenvalue $\lambda$, and $(\lambda,v)$ is an eigenpair of $A$.

\medskip
\begin{itemize}

\item Let $\sigma (A)$ be the multiset of the eigenvalues of $A$, i.e., the \emph{spectrum} of $A$.
\item Let $V_{\lambda}(A)=\{x\in \mathbb{C}^n : Ax=\lambda x\}$ be the \emph{eigenspace} of $A$ associated with $\lambda$.

\end{itemize}

It's obvious that for each eigenvalue  $\lambda$,
 $V_{\lambda}(A)=\{0\}\cup \{x \in \Bbb C^n : \mbox{$(\lambda, x)$ is an eigenpair of $A$}\}$.
 \medskip

We begin with a basic, ``simple'' question: If   $A$ and $B$  have exactly the same eigenvalues
(the spectrum) and the same corresponding eigenvectors (the eigenspace), does $A=B$?
The answer is negative in general.

\begin{exa}\label{exa:1}
\rm The matrices
 $A=\begin{bmatrix} 1 & 1  \\0 & 1\end{bmatrix}$  and
$B= \begin{bmatrix} 1 & 2\\ 0 & 1\end{bmatrix}$
  have the same eigenvalues $\lambda_1=\lambda_2=1$, and the same
 associated eigenvectors in the form  $ \begin{bmatrix} c \\ 0\end{bmatrix}$,
 $c\not =0$.
\end{exa}


This example shows that having the same eigenpairs does not uniquely determine the matrix in general  (unless the eigenvectors  form a basis).
\medskip

Note that in this case the two matrices are similar. If we assume a bit more, they are forced to be equal. For example, assume that $A$ is  diagonalizable. If $B$ has exactly the same eigenpairs as $A$, then $A=B$; if $B$ has exactly the same eigenvectors (regardless of the eigenvalues) as $A$, then $B$ is diagonalizable.
\medskip

%
%
%
%
%

\begin{exa}\rm
Consider the following two \(4 \times 4\) matrices:
\[
A = \begin{bmatrix}
0 & 0 & 1 & 0 \\
0 & 0 & 0 & 1 \\
0 & 0 & 0 & 0 \\
0 & 0 & 0 & 0
\end{bmatrix}\quad \mbox{and}\quad
B = \begin{bmatrix}
0 & 0 & 1 & 0 \\
0 & 0 & 0 & 0 \\
0 & 0 & 0 & 1 \\
0 & 0 & 0 & 0
\end{bmatrix}.
\]

The only eigenvalue  of $A$ and $B$ is 0
 with multiplicity 4. For $A$, solving $Ax=0$ with $x=(x_1, x_2, x_3, x_4)^T$  gives $x_3=0$, $x_4=0$, so $x=(x_1, x_2, 0, 0)^T$.
For $B$, solving $Bx=0$ gives $x_3=0$ and $x_4=0$, so again $x=(x_1, x_2, 0,0)^T$. Thus, $A$ and $B$  have the same eigenspace, i.e., the same eigenpair.
\medskip

However, $A$ and $B$ are not similar. Note that
$A^2 = 0$, 
while $B^2 \neq 0$.
Therefore, having the same eigenpairs is not enough to imply the similarity of the
two matrices.
\end{exa}

What happens if $B=A^T$, the transpose of $A$, or $B=A^*$, the conjugate transpose of $A$. As remarked earlier, from the characteristic polynomial we see that the  eigenvalues of $A^T$ are exactly the eigenvalues of $A$ (i.e., $\sigma (A^T)=\sigma (A)$);
the eigenvalues of $A^*$ are the conjugates of the eigenvalues of $A$
 (i.e.,  $\sigma (A^*)=\overline{\sigma (A)}$\,).  However,
the eigenvectors of $A^T$ (and $A^*$) can be different from those of $A$.

\begin{exa}\rm
 Let  $A=\begin{bmatrix} 1 & 1 \\0 & 2\end{bmatrix}$. Then
$A^T=A^*=\begin{bmatrix} 1 & 0 \\1 & 2\end{bmatrix}$.
The eigenvalues of $A$ (the same as those of $A^T$)  are
$\lambda_1=1$ and $\lambda_2=2$. By computation, we have

$$V_{1}(A)\ = \ \left \{ c \begin{bmatrix} 1 \\ 0\end{bmatrix} : c\in \Bbb C\right  \}, \quad
V_{2}(A)\ = \ \left  \{ c \begin{bmatrix} 1 \\ 1 \end{bmatrix} : c\in \Bbb C\right \} $$
and
$$V_{1}(A^T)\ = \ \left \{ c \begin{bmatrix} 1 \\ -1\end{bmatrix} : c\in \Bbb C\right  \}, \quad
V_{2}(A^T) \ = \ \left  \{ c \begin{bmatrix} 0 \\ 1 \end{bmatrix} : c\in \Bbb C\right \}.$$
Any two of these eigenspaces contain no nonzero eigenvectors in common.
\end{exa}

In fact, more  is known about $A$ and $A^T$ (and $A^*$) in terms of similarity:
$A^T$ is always similar to $A$ (via Jordan blocks) (see, e.g.,
\cite[Theorem\,3.5.1, p.\,153]{Zh});  $A^*$ is similar to $A$ if and only if
$A$ is similar to a real matrix and if and only if $A$ is a product of two Hermitian matrices
(see,   e.g., \cite[Theorem\,8.2.2, p.\,351]{Zh}).
\medskip


We can now prove our characterization of those matrices  $A$ for which  $A$ and $A^T$ (or $A^*$) have exactly the same set of  eigenvectors (regardless of eigenvalues) and exactly the same eigenpair (equivalently, the same eigenspace) for each eigenvalue. We assume that the order of the matrix $A$ is $n$.

\subsection{Proof of Theorem \ref{Thm:Normality}} We restate our first main result, and then give the proof.  We assume that the reader is familiar with the basic spectral decomposition theorems from a first course in linear algebra.
\medskip

\noindent \textbf{Theorem \ref{Thm:Normality}.} \emph{Every eigenvector  of $A$ is an eigenvector (regardless of the eigenvalues)
of $A^*$ if and only if $A$ is normal,
that is,  $A^*A=AA^*$.}



\begin{proof}  By the spectral decomposition theorem (see,   e.g., \cite[Theorem\,3.2.2, p.\,121]{Zh}), a square matrix is normal if and only if it's unitarily diagonalizable.
\medskip

If $A$ is normal, then there exists a unitary matrix\footnote{A matrix $U$ is unitary if $UU^* = U^*U = I$.} $U$ such that
$A=U^*DU$, where $D$ is a diagonal matrix whose diagonal entries are the eigenvalues of $A$.
Thus,  $A^*=U^*D^*U$. Note that  
 $$Av \ = \ \lambda v \  \ \Leftrightarrow \  \ U^*DU v\ = \ \lambda v  \ \  \Leftrightarrow \  \  Dw \ =\ \lambda w$$
 (where $w=Uv$). Continuing, the above is true if and only if $$\overline{D}w \ = \
 \overline{\lambda} w \
 \   \Leftrightarrow  \ D^*Uv \ = \
 \overline{\lambda} Uv    \ \Leftrightarrow
 \   U^*D^*Uv \ = \
 \overline{\lambda} U^*Uv    \ \Leftrightarrow  \
  A^*v\ =\
 \overline{\lambda} v.$$  
 So, for normal $A$,   $(\lambda, v)$ is an eigenpair of $A$ if and only if  $(\overline{\lambda}, v)$ is an eigenpair of $A^*$.
 \medskip

For the other direction, suppose that
$Av = av$ \mbox{for some $a\in \Bbb C$  and $v\in \Bbb C^n$}
  implies  that $A^*v = bv $ \mbox{for some $b\in \Bbb C$}.
We show that $A$ is a  normal matrix.
\medskip

Without loss of generality,  we may rescale and assume that the eigenvector $v$ has length 1. Note that
$$v^*Av \ = \  a v^*v \ = \ a.$$ Taking complex conjugates yields $$\bar{a} \ = \ (v^*Av)^* \ = \  v^*A^*v \ = \ bv^*v\ = \ b.$$
Thus,  if $(a, v)$ is an eigenpair of $A$, then $(\bar{a}, v)$ is an eigenpair of $A^*$.
\medskip

We complete the proof by showing that $A$ is unitarily diagonalizable by induction on $n$.
\medskip

If $n=1$, there is nothing to show. Suppose that the statement holds for matrices of order less than $n$. We reduce the normality of $A$ to that of a matrix of order $n-1$ that inherits the eigenpair property of $A$.
\medskip

Let $v_1=v, v_2, \dots, v_n$ be an orthonormal basis of $\Bbb C^n$. Then \begin{equation} \twocase{v_j^*Av_1 \ = \ }{a}{if $j=1$}{0}{otherwise,} \end{equation} which implies \begin{equation} \twocase{v_j^*A^*v_1 \ = \ }{\bar{a}}{if $j=1$}{0}{otherwise.} \end{equation} Let $V=(v_1, v_2, \dots, v_n)$. Then $V$ is a unitary matrix (as the $v$'s are an orthonormal basis of $\Bbb C^n$).
\medskip

Note that
$$AV \ = \ (Av_1, Av_2, \dots, Av_n) \ = \ (a v_1, Av_2, \dots, Av_n).$$ The first column of $V^*AV$,
consisting of all  $v^*_jAv_1$,     is
$\begin{bmatrix}
a   \\
0 \end{bmatrix}$ (here 0 means we have an entry of zero a total of $n-1$ times). \\ \

In a similar way, the first column of $V^*A^*V$ is $ V^*A^*v_1=
 \begin{bmatrix}
\bar{a}  \\
0 \end{bmatrix}$. Because $(V^*AV)^*=V^*A^*V$ and  the first column of $V^*A^*V$ is the conjugate transpose of the first row of $V^*AV$,
 we see that $V^*AV$ has the form
$$
V^*AV \ =  \  \begin{bmatrix}
a & 0 \\
0 & A_1 \end{bmatrix}$$ for   
some matrix $A_1$ of order $n-1$. As every eigenvector of $A$ is an
eigenvector of $A^*$,   every eigenvector of $A_1$ is an
eigenvector of $A^*_1$.
By the induction hypothesis, $A_1$ is normal. Thus,  $A$ is unitarily similar to a  normal matrix, and $A$ is normal.
\end{proof}

The above theorem  is Theorem 9.1.1\,(6) of \cite{Zh}; the proof there is slightly different.

\begin{rek} The statement that if $A$ is normal then every eigenvector  of $A$ is an eigenvector of $A^*$ is also shown as follows. For normal $A$,  $A^*$ can be expressed as a polynomial in $A$, say $A^*=p(A)$ for some polynomial $p$ (see, e.g., Theorem 9.1.1\,(3) of \cite{Zh}). It follows that $A^*v=p(A)v =p(\lambda )v$ if $Av=\lambda v$. \end{rek}

Theorem\,\ref{Thm:Normality}  can be simply restated as: \emph{$A$ is normal if and only if $A$ and $A^*$ have the same eigenvectors (possibly corresponding to different eigenvalues).}
\medskip

The next example shows that $A$ need not be symmetric.

\begin{exa}\rm  The matrix $A=\begin{bmatrix} 1 & i \\-i & 1\end{bmatrix}$ is Hermitian (normal). While $A$ and $A^T$ have exactly  the same eigenvectors, $A\not =A^T$. Note that $A$ and $A^*(=A)$ have the same eigenvalues, the same eigenspaces (and the same eigenpairs), while $A$ and $A^T$ have the same eigenvalues but different eigenspaces.
 \end{exa}

The next example shows that $A$ need not be Hermitian.

\begin{exa}\label{Exa:4notHer}
\rm We give an example where $A$ and $A^*$ have exactly  the same eigenvectors, but $A\not =A^*$. Let $A=\begin{bmatrix} 1 & i \\i & 1\end{bmatrix}$; the matrix is normal and
 has eigenvalues $\lambda_1 =1+i$ and $\lambda_2 =1-i$ with
corresponding eigenvectors $v_1=\begin{bmatrix}  1 \\ 1 \end{bmatrix}$ and
$v_2=\begin{bmatrix}  1 \\ -1 \end{bmatrix}$ (up to scalar multiples), respectively. The eigenpairs of $A$ are
$(\lambda_1, v_1)$ and $(\lambda_2, v_2)$. It follows that
$$V_{\lambda_1}(A) \ = \ \left \{ c\begin{bmatrix} 1 \\ 1\end{bmatrix} :c\in \Bbb C\right  \}, \quad
V_{\lambda_2}(A) \  = \  \left  \{ c\begin{bmatrix} 1 \\ -1 \end{bmatrix} :c\in \Bbb C\right \}.$$ We find that $A^*=\begin{bmatrix} 1 & -i \\-i & 1\end{bmatrix}$  has eigenvalues  $\mu_1 =1+i$ and $\mu_2 =1-i$ with
corresponding eigenvectors $u_1=\begin{bmatrix}  1 \\ -1 \end{bmatrix}$ and
$u_2=\begin{bmatrix}  1 \\ 1 \end{bmatrix}$ (up to scalar multiples), respectively.
The eigenpairs of $A^*$ are
$(\mu_1, u_1)$ and $(\mu_2, u_2)$.  It follows that
$$V_{\mu_1}(A^*) \ = \ \left \{ c\begin{bmatrix} 1 \\ -1\end{bmatrix} :c\in \Bbb C\right  \}, \quad
V_{\mu_2}(A^*) \ = \ \left  \{ c\begin{bmatrix} 1 \\ 1\end{bmatrix} : c\in \Bbb C\right \}.$$
Note that  this example shows that  every eigenvector of $A^*$ is an eigenvector of $A$ (associated with different eigenvalues), that is, $A$ and $A^*$ have exactly the same eigenvectors, but different eigenpairs (thus different eigenspaces).
\end{exa}

\subsection{Proof of Theorem \ref{Thm:Normality2}} We end with a proof of our second main result. \\ \

\noindent \textbf{Theorem \ref{Thm:Normality2}.} \emph{If every eigenpair of $A$ is an eigenpair of $A^*$; to be exact,
\begin{equation*}
Av = \lambda v
 \quad \Rightarrow \quad
A^*v = \lambda v,
\end{equation*}
then $A$ is a Hermitian matrix, that is, $A^*=A$ (thus all eigenvalues are real).}

\begin{proof}  By Theorem \ref{Thm:Normality}, $A$ is normal. Thus, $A$ possesses $n$     orthonormal eigenvectors, say
$v_1, \dots, v_n$,  which  form a basis for $\Bbb C^n$. We have
 $(A-A^*)v_k=Av_k-A^*v_k =0$, $k=1, 2, \dots, n$. So, $A-A^*=0$ and $A$ is Hermitian. \end{proof}

Theorem \ref{Thm:Normality2} can be restated simply as: \emph{$A$ is a Hermitian matrix if and only if $A$ and $A^*$ have the same eigenpairs for every eigenvalue (of $A$).}
\medskip

In fact, if $\lambda$ is an eigenvalue of $A$, then
$\overline{\lambda}$ is an eigenvalue of $A^*$.
In view of  Theorem~\ref{Thm:Normality2}, it's  tempting to believe the following statement is true: \emph{If  for every eigenpair, $(\lambda ,v)$,  of $A$,  $(\bar{\lambda},v)$ is an eigenpair of $A^*$, then $A$ is Hermitian, i.e., $A^*=A$.} However, Example \ref{Exa:4notHer} shows that this statement is false.

\begin{corollary}\label{Cor:1}
If $A$ is a real matrix with real eigenvalues such that  every   eigenvector of $A$ is an eigenvector of $A^T$,
then $A$ is symmetric, i.e., $A^T=A$.
\end{corollary}

\begin{proof}  For real $A$, the condition  states that every eigenvector of $A$ is an eigenvector of $A^T=A^*$ (regardless of the eigenvalues). Theorem \ref{Thm:Normality}   ensures that $A$ is normal. If all eigenvalues of $A$ are real, then $A$ is Hermitian.
It follows that  $A^T=A^*=A$.
\end{proof}

Note that for a real matrix, if an eigenvalue of the matrix is real, then the corresponding eigenvectors are real. Moreover, the condition that all eigenvalues of $A$  be real in the above corollary is needed. See Example~\ref{Exa:3by3} below.
\medskip

If $A$ is nonreal, then it's possible that $A$ is nonnormal.

\begin{exa}\rm
Let $
A=
\begin{bmatrix}
1 & i\\
i & -1
\end{bmatrix}.
$
Then
$
A^T=A,
$
so $A$  and $A^T$ have exactly the same eigenpairs.
It's easy to check that $A$ is not normal.
\end{exa}

\begin{corollary}\label{Cor:4}
\rm If $A$ is a real matrix such that
every (possible complex) eigenpair    of $A$
 is an eigenpair of $A^T$,  then $A$ is symmetric, i.e., $A^T=A$.
\end{corollary}

\begin{proof} This is immediate from Theorem \ref{Thm:Normality2} as $A^T=A^*$. \end{proof}

In   Corollary \ref{Cor:4}, the eigenvalues and eigenvectors are possibly complex (even though $A$ is real). In other words,
$A$ and $A^T$ need not be equal if $A$ and $A^T$ just have  the same {\em real} eigenvectors (or real eigenpairs).

\begin{exa}\label{Exa:3by3}
\rm   Let
$A=\begin{bmatrix}
\ \ 0 & 1 &0  \\
-1 & 0 & 0\\
\ \ 0 & 0 & 1
\end{bmatrix}$. Obviously, $A\not = A^T$.
One checks that
$A$ is   normal,  and $A$ and $A^T$ have all the
same {real} eigenvectors in the form $\begin{bmatrix}
0  \\
0\\
r
\end{bmatrix}$, $0\not = r\in \Bbb R$,  corresponding to the eigenvalue $1$. Moreover, the eigenpairs of $A$ and $A^T$ are respectively as follows:
\medskip

\noindent
  $\lambda_1(A) =i$ with corresponding eigenvectors in the form  $$u_1 \ = \  \begin{bmatrix}
1  \\
i\\
0
\end{bmatrix}, \ \ \  0\not = c\in \Bbb C,$$
  $\lambda_2(A) =-i$ with the corresponding eigenvectors in the form  $$ u_2 \ = \ c \begin{bmatrix}
1  \\
-i\\
0
\end{bmatrix}, \ \ \  0\not = c\in \Bbb C; $$
$\lambda_1(A^T) =i$ with corresponding eigenvectors in the form  $$v_1 \ = \  \begin{bmatrix}
1  \\
-i\\
0
\end{bmatrix}, \ \ \ 0\not = c\in \Bbb C,$$
$\lambda_2(A^T) =-i$ with corresponding eigenvectors in the form  $$v_2 \ = \  \begin{bmatrix}
1  \\
i\\
0
\end{bmatrix}, \ \ \ 0\not = c\in \Bbb C.$$
\end{exa}

\begin{exa}
\rm Consider
$A=\begin{bmatrix}
0 & 1 & i \\
-1 & 0 & 1\\
-i & -1 &0
\end{bmatrix}$, which is skew-symmetric, i.e.,
$A^T=-A$.  One checks that
  $A$ is not normal (nor symmetric). 
We see
$Av=\lambda v$ if and only if $A^T v=-\lambda v$. So, $A^T$ and $A$ have the same eigenvectors (not the same eigenpairs). Let
$$B \ = \ \begin{bmatrix}
1 & 0 \\
0 & A
\end{bmatrix}\ = \ \begin{bmatrix}
1 & 0 & 0 & 0\\
0 & 0 & 1 & i \\
0 & -1 & 0 & 1\\
0 & -i & -1 &0
\end{bmatrix}.$$ Then $B$ is not normal, symmetric, or skew-symmetric, and $B^T$ has
exactly the same eigenvectors as $B$.
\end{exa}




\section*{About the authors:}

\subsection*{Kayla Miller}
Mount Greylock Regional School and Williams College, Williamstown, MA 01267: ksmill5433@gmail.com.

Kayla is a senior at Mount Greylock Regional School and a research affiliate of Williams College. This is her third
paper. She is the president of the Spanish Club at her high school. She enjoys gymnastics, reading, card games, and
spending time with friends and family. 

\subsection*{Steven J. Miller}

Williams College, Williamstown, MA 01267: sjm1@williams.edu.

Steven graduated with a B.S. in mathematics and physics from Yale and a Ph.D. in
mathematics from Princeton. He has supervised more than 600 students and
written more than 200 papers in accounting, computer science, economics,
geophysics, marketing, mathematics, operations research, physics,
sabermetrics, and statistics. He is a founding member of the Polymath Jr.
Summer Research Program and the President of the Fibonacci Association. He
loves multitasking, preferably with his family, who have published a few
papers with him and created some interesting bridge problems. He can also
teach people how to poorly solve a Rubik's cube in under an hour.

\subsection*{Fuzhen Zhang}~
Nova Southeastern University, Fort Lauderdale, Florida 33314: zhang@nova.edu.

Fuzhen received his Ph.D. in Mathematics from the University of California - Santa Barbara (UCSB). He is currently a
professor at Nova Southeastern University (NSU-Florida) in Fort Lauderdale, Florida, USA. With primary research interests in matrix analysis and linear and multilinear algebra, Dr. Zhang has published more than 90 research articles and delivered over 120 talks worldwide. He is a recipient of the NSU President's Distinguished Professor of the Year Award and the Shanghai City ``Overseas
Renowned Professors'' Award. His hobbies include playing beach volleyball, pickleball, and singing Chinese folk songs. 

\end{document}